\documentclass{amsart}
\usepackage{amssymb, amsmath}
\usepackage{amsthm}
\usepackage{quiver}
\usepackage{enumitem}
\usepackage[backend=biber,style=numeric,sorting=nyt,url=false,doi=false,isbn=false]{biblatex}
\newcommand{\Hom}{\operatorname{Hom}}
\newcommand{\Spec}{\operatorname{Spec}}
\newcommand{\HS}{\operatorname{HS}}
\newcommand{\FL}[1]{\operatorname{FL}(#1)}

\title{Finiteness of leaps of modules of integrable derivations of algebras of finite type}
\author{Takuya Miyamoto}
\address{Mathematical Sciences, the University of Tokyo, Meguro Komaba 3-8-1, Tokyo, Japan}
\email{miyamoto-takuya@g.ecc.u-tokyo.ac.jp}

\begin{document}

\maketitle

\begin{abstract}
We prove the finiteness of leaps of modules of $m$-integrable derivations for algebras essentially of finite type and, more generally, for schemes essentially of finite type over an algebraically closed field of positive characteristic. This provides an affirmative answer to a question posed by L. {Narváez Macarro}. As an application, we establish the coherence of the module of $\infty$-integrable derivations.
\end{abstract}

\tableofcontents

\newtheorem{theorem}{Theorem}[section]
\newtheorem{proposition}[theorem]{Proposition}
\newtheorem{corollary}[theorem]{Corollary}
\newtheorem{lemma}[theorem]{Lemma}
\theoremstyle{definition}
\newtheorem{definition}[theorem]{Definition}

\newtheorem{assumption}[theorem]{Assumption}
\newtheorem{setting}[theorem]{Setting}
\newtheorem{deflemma}[section]{Definition-Lemma}
\theoremstyle{remark}
\newtheorem{remark}{Remark}[theorem]
\newtheorem{example}{Example}[theorem]
\theoremstyle{definition}

\section{Introduction}
Hasse-Schmidt derivations (abbreviated as HS-derivations) were introduced in \cite{SchmidtHasse+1937+215+237}, which are also called higher derivations \cite[\S 27]{Mat}. Given a positive integer $m$, a field $k$ and a $k$-algebra $R$, $\HS _k^m(R)$ denotes the group of HS-derivations of length $m$, which consists of $A_m:=k[\![t]\!]/(t^{m+1})$-algebra automorphisms
$$R[\![t]\!]/(t^{m+1}) \to R[\![t]\!]/(t^{m+1})$$
that become the identity when restricted to $ R[\![t]\!]/(t^{m+1}) \otimes_{A_m} k=R$. For $m=\infty$, we regard $t^\infty=0$, and we also have the notion of HS-derivations of infinite length. There is a canonical isomorphism between $\HS _k^1(R)$ and the module $\operatorname{Der}_k(R,R)$. Under this identification, the image of
$$\HS _k^m(R)\to \HS _k^1(R)\simeq \operatorname{Der}_k(R,R)$$
is called the module of $m$-integrable derivations and denoted by $\operatorname{Der}_k^m(R)$, which has been studied in \cite{matsumura1982integrable}, \cite{NARVAEZMACARRO20122712}, \cite{NARVAEZMACARRO2024109758}. The module $\operatorname{Der}_k^\infty(R)$ is also denoted by $\operatorname{Ider}_k(R)$. Note that if $k$ contains $\mathbb{Q}$ or if $R$ is formally smooth over $k$, we have $\operatorname{Der}_k^m(R)=\operatorname{Der}_k(R)$ for any $m\le \infty$ (see \cite{Mat}). Thus, the study of integrable derivations is expected to lead to a better understanding of singularities in positive characteristic. From now on, we assume that $k$ is an algebraically closed field of positive characteristic $p$ and $R$ is essentially of finite type. We have the following descending chain
$$ \operatorname{Der}_k(R)\supset \operatorname{Der}_k^2(R)\supset \dots \supset \operatorname{Der}_k^m(R) \supset\dots.$$
The question of when the inclusion $\operatorname{Der}_k^{m-1}(R)\supset \operatorname{Der}_k^m(R)$ is proper has been studied since their introduction. When this is the case, we say that $R$ {\em leaps} at $m$. It has been known that leaps occur only at $m=p^i$ (\cite[Theorem 4.1]{Her}). The number of leaps is finite if $R$ is a unibranch curve (\cite[Theorem 4.7]{NARVAEZMACARRO2024109441}), and, more generally, if $R$ is a reduced curve (\cite[Theorem 6.1, Corollary 6.3]{bravo2024finitenessleapssensehasseschmidt}).
The aim of this paper is to show, using results from \cite[\S 2]{miy2025}, that the finiteness of leaps holds for every algebra essentially of finite type over an algebraically closed field of positive characteristic. Precisely, we show:

\begin{theorem}[{= Theorem \ref{infty integrable thm}}]\label{main thm1}
    There exists $M>0$ (depending on $R$) such that
$$ \operatorname{Der}_k^M(R)= \operatorname{Der}_k^{M+1}(R)=\dots= \operatorname{Der}_k^\infty(R)=\operatorname{Ider}_k(R).$$
\end{theorem}

This theorem provides, for the first time, a general finiteness result for leaps in arbitrary dimensions, thereby giving an affirmative answer to a question posed by L. {Narváez Macarro} [8, Question 3.6.5] and contributing to a better structural understanding of HS-derivations. Since leaps are stable under separable base changes by \cite[Theorem 3.27]{hernandez2020behavior}, this also implies the finiteness of leaps over any perfect field. We also give an affirmative answer to another question \cite[Question 3.6.1]{NARVAEZMACARRO20122712} by the following:

\begin{theorem}[{= Theorem \ref{mn int theorem}}]\label{main thm2}
    Let $n\ge 1$ be an integer and let $D\in \HS_k^n(R)$. Then $D$ is $\infty$-integrable if and only if $D$ is $m$-integrable for every integer $m\ge n$. Here, we say that $D$ is $m$-integrable if $D$ is in the image of $\HS_k^m(R)\to \HS_k^n(R)$.
\end{theorem}
In fact, Theorem \ref{main thm1} is a special case of Theorem \ref{main thm2} where we put $n=1$. As an application of Theorem \ref{main thm1}, we obtain:

\begin{corollary}[{=Corollary \ref{localization ider cor}}]\label{main cor}
The module $\operatorname{Der}_k^\infty(R)=\operatorname{Ider}_k(R)$ is compatible with localization of $R$. That is, we have $$\operatorname{Ider}_k(S^{-1}R)=S^{-1}\operatorname{Ider}_k(R).$$
\end{corollary}
This is the first coherence result for $\operatorname{Ider}_k(R)$ to the best of our knowledge. For $m<\infty$, a similar statement has been proven in \cite[Corollary 2.3.5]{NARVAEZMACARRO20122712}. The case $m=\infty$ is more subtle because the map in \cite[Corollary 1.3.6]{NARVAEZMACARRO20122712} is not surjective for $m=\infty$ while $m$ is assumed to be an integer there (cf. \cite[Example 1.4]{MR1764573}). Note that Corollary \ref{main cor} implies that $\operatorname{Ider}_k(R)$ constitutes a coherent sheaf $\operatorname{Ider}_k(\mathcal{O}_X)$ on an algebraic scheme $X$ over $k$. In particular, it is expected to behave better than the usual tangent sheaf. In fact, some pathological phenomena of derivations peculiar to positive characteristic do not occur if we replace it by integrable derivations. For example, the local rank of $\operatorname{Ider}_k(R)$ is at most $\operatorname{dim}R$, while it is not true for $\operatorname{Der}_k(R)$ (see \cite{MR554764}). This pathology does not occur in characteristic zero (see \cite{MR366909} when $R$ is a domain and \cite{MR498523} in general). We can also state Theorem \ref{main thm1} in terms of schemes:
\begin{theorem}[{= Theorem \ref{infty integrable thm schemes}}]\label{main thm3}
    Let $X$ be a $k$-scheme essentially of finite type (=locally essentially of finite type and quasi-compact). Then there exists $M>0$ such that
$$ \operatorname{Der}_k^M(\mathcal{O}_X)= \operatorname{Der}_k^{M+1}(\mathcal{O}_X)=\dots=\operatorname{Ider}_k(\mathcal{O}_X),$$
where $\operatorname{Der}_k^m(\mathcal{O}_X)$ denotes the sheaf of $m$-integrable derivations (see Theorem \ref{localization}).
\end{theorem}

We briefly outline the contents of the paper.

In Section 2, we recall the definition of HS-derivations and results that will  be used later. A key relation is that the obstructions to extending HS-derivations are canonically dual to $m$-integrable derivations (Theorem \ref{duality}). We also define the property $\FL{i}$ for algebras and schemes, which states that leaps do not occur at every $m> p^i$ (Definition \ref{FL def}).

In Section 3, we introduce the notion of $a$-boundedness for HS-derivations. Intuitively, this concept measures how each term of a HS-derivation shrinks with respect to the $\mathfrak{m}$-adic topology.

In Section 4, we prove the finiteness of leaps in the local case. In this section we fix a local $k$-algebra $(R,\mathfrak{n})$ such that $\operatorname{Spec}R\setminus \{\mathfrak{n}\}$ satisfies $\FL{\tau}$ for $\tau>0$, and we will prove that $\operatorname{Spec}R$ itself satisfies the finiteness of leaps. The idea is that, starting from a bounded HS-derivation, we modify it step by step so that the obstruction at each stage becomes zero while the lower terms remain unchanged. We use the Artin-Rees lemma to obtain positive constants which control the bounds. The essential step is Theorem \ref{eq finie thm}, which states that elements of $$\mathfrak{n}^{e(q+1)+1} \cdot \operatorname{Der}_k^{q}(R)$$ do not leap, where $q=p^\tau$ and $e\gg 0$. Then the fact that $R/\mathfrak{n}^{e(q+1)+1} $ is of finite length implies the finiteness of leaps of $R$ (Theorems \ref{Her thm}, \ref{sec 4 main thm}).

In Section 5, we prove the finiteness of leaps in the global case and discuss its consequences. We use noetherian induction to reduce to the case of Section 4 (Theorem \ref{main thm global}). The finiteness of leaps of $m$-integrable derivations implies that of obstruction modules. This enables us to extend HS-derivations step by step (Lemma \ref{mn int lem}), and taking the limit, we obtain a HS-derivation of infinite length: This method yields a characterization of $\infty$-integrability (Theorem \ref{mn int theorem}).

Throughout this paper, we fix an algebraically closed field $k$ of positive characteristic $p$. We assume that any $k$-algebra is unital and commutative. We say that a $k$-scheme $X$ is essentially of finite type if there exists a finite open affine covering $\{U_i=\operatorname{Spec}R_i\}_{i=1,\dots,n}$ for $X$ where each $R_i$ is a $k$-algebra essentially of finite type. We assume $\mathbb{N}:=\{0,1,\dots\}$.

\section{Preliminaries}

We first recall basic facts of HS-derivations (cf. \cite{SchmidtHasse+1937+215+237}, \cite[Chapter 1]{NARVAEZMACARRO20122712}, \cite[\S 27]{Mat}) together with results on $m$-integrable derivations and obstruction modules (cf. \cite{NARVAEZMACARRO20122712}, \cite{Her}, \cite[\S 2]{miy2025}). The obstruction modules govern the extension of HS-derivations from length $m-1$ to $m$. These obstructions form coherent sheaves on schemes essentially of finite type (Theorem \ref{Ob sheaf}). Moreover, the ascending filtration of obstruction modules is naturally dual to the descending chain of
$m$-integrable derivations (Theorem \ref{duality}), a relation that plays a key role in establishing the finiteness of leaps.

\begin{definition}[= {\cite[Definition 1.2.1-1.2.6]{NARVAEZMACARRO20122712}}] \label{Hasse Schmidt}
Let $R$ be a $k$-algebra. Let $m\ge 1$ be an integer or $m=\infty$. Then a \emph{Hasse-Schmidt derivation} (abbreviated as HS-derivation) of length $m$ is a sequence $D=(D_i)_{i=0}^m$ of $k$-linear endomorphisms $R\to R$ such that $D_0=\operatorname{id}$ and that $$D_i(xy)= \sum_{j+k=i}D_j(x)D_k(y)$$ for every $x,y\in R$, $ i\ge 1.$ Let $E=(E_i)_{i=0}^m$ be another HS-derivation. We define the \emph{composition} of $D$ and $E$ by
$$D\circ E=(F_i)_{i=0}^m \mathrm{\ \ where \ } \ F_i(x)=\sum_{j+k=i}D_j(E_k(x)).$$
The symbol $\operatorname{HS}_k^m(R)(=\operatorname{HS}_k(R,m))$ denotes the set of all the HS-derivations of length $m.$ It forms a group under this compsition. Let $x\in R$. We define 
$$ x \bullet: \operatorname{HS}_k^m(R)\to\operatorname{HS}_k^m(R)$$
by
$$x\bullet (D_i)_{i=0}^m=(x^iD_i)_{i=0}^m.$$
For $n\ge1,$ we define
$$[n]:\operatorname{HS}_k^m(R)\to\operatorname{HS}_k^{mn+n-1}(R)$$
by
$$D[n]=(E_i)_{i=0}^{mn+n-1}\mathrm{\ \ where\ }\ E_{jn}=D_j \mathrm{\ for\ }0\le j\le m \mathrm{\ and \ } E_i=0\mathrm{\ otherwise.}$$
(Although $D[n]$ is usually defined to have length $mn$, it naturally extends to length $mn+n-1$ by adding zero terms.)
A HS-derivation $D=(D_i)_{i=0}^m$ defines a natural $A_m$-algebra automorphism $$R[\![t]\!]/(t^{m+1})\to R[\![t]\!]/(t^{m+1}),\ x\mapsto \sum_{i=0}^m D_i(x) t^i\quad \mathrm{for}\ x\in R$$ 
where $A_m:=k[\![t]\!]/t^{m+1}$. (We regard $t^\infty=0$.) This yields an isomorphism of groups between $\HS_k^m(R)$ and the group of $A_m$-algebra automorphisms of $R[\![t]\!]/(t^{m+1}) $
that become the identity when restricted to $ R[\![t]\!]/(t^{m+1})\otimes_{A_m} k=R$.
Let $1\le l \le m$. We define the \emph{truncation map} $$\operatorname{HS}_k^m(R)\to \operatorname{HS}_k^l(R)$$ by $$D=(D_i)_{i=0}^m\mapsto(D_i)_{i=0}^l.$$ 
It is a group homomorphism and commutes with $x\bullet$ for $x\in R$.
Note that for $m=1$, $\operatorname{HS}_k^1(R)$ identifies with the $R$-module $\operatorname{Der}_k(R,R)$ under $$D=(\operatorname{id},D_1)\mapsto D_1.$$
\end{definition}

\begin{definition}[= {\cite[Definition 2.1.1]{NARVAEZMACARRO20122712},\ \cite[\S 1]{matsumura1982integrable}}] \label{m- inf- integrable der}Let $R$ be a $k$-algebra and let $1\le m \le \infty$. We define the module of\emph{ $m$-integrable derivations }$\operatorname{Der}_k^m(R)$ as the image of
$$\operatorname{HS}_k^m(R)\overset{}{\to}\operatorname{HS}_k^1(R)\simeq \operatorname{Der}_k(R,R).$$ It is a sub-$R$-module of $\operatorname{Der}_k(R,R)$. We also denote $\operatorname{Der}_k^\infty(R)$ by $\operatorname{Ider}_k(R)$.
\end{definition}

Let $d\in \operatorname{Der}_k(R,R)$. For $1< m< \infty$, we say that $d$ leaps at $m$ if $d\in \operatorname{Der}_k^{m-1}(R)$ and $d\notin \operatorname{Der}_k^m(R)$. If $d\in \cap_{1<m<\infty} \operatorname{Der}_k^m(R)$, then $d$ does not leap. Let $S\subset \operatorname{Der}_k(R,R)$ be a subset. By $\operatorname{Leaps}(S)\subset \mathbb{N}$, we denote the set of leaps produced by $S$. The set of leaps of $R$ is $\operatorname{Leaps}(\operatorname{Der}_k(R,R))$.

\begin{theorem}[= {\cite[Corollary 2.3.5]{NARVAEZMACARRO20122712}}]
\label{localization}
    Let $R$ be a $k$-algebra essentially of finite type over $k$. Let $m\ge 1$ an integer. Let $S\subset R$ a multiplicatively closed subset. Then we have a canonical isomorphism $$\alpha: S^{-1}\operatorname{Der}_k^m(R)\overset{\sim}{\to} \operatorname{Der}_k^m(S^{-1}R).$$
\end{theorem}

This implies that if $m<\infty$ and $X$ is a $k$-scheme essentially of finite type, its $m$-integrable derivations consist a sheaf on $X$, which is denoted by $\operatorname{Der}_k^m(\mathcal{O}_X)$.

\begin{theorem}[= {\cite[Theorem 4.1]{Her}}]\label{p^n leap} Let $R$ be a $k$-algebra essentially of finite type. For every integer $m\ge 2,$ the inclusion $\operatorname{Der}_k^m(R)\subset \operatorname{Der}_k^{m-1}(R) $ is proper only if $m$ is of the form $m=p^i$ for $i\ge 1$.
\end{theorem}

Let us recall the {\em first cotangent module} $T^1_{R/k}$ of a $k$-algebra $R$ essentially of finite type (cf. \cite[Definition 1.1.6]{Ser}). Let $T$ be a localization of a polynomial algebra over $k$ and $I\subset T$ an ideal such that $T/I=R.$ By \cite[Corollary 1.1.8]{Ser}, we have the following canonical exact sequence of $R$-modules

$$\operatorname{Hom}_R(\Omega_{T/k}\otimes_TR,R)\to \operatorname{Hom}_R(I/I^2,R)\to T^1_{R/k}\to 0.$$

\begin{definition}[{= \cite[Definition 2.11]{miy2025}}]\label{obstruction sp M,Ob}
    Let $R$ be a $k$-algebra essentially of finite type and $m\ge 2 $ an integer. We define $$\operatorname{ob}_m:\operatorname{HS}_k^{m-1}(R)\to T^1_{R/k}$$
    as follows: Let $T$ be a localization of a polynomial algebra over $k$ and $I\subset T$ an ideal such that $T/I=R.$ For $D\in \operatorname{HS}_k^{m-1}(R)$, suppose that $E\in \operatorname{HS}_k^m(T)$ is a lift of $D$. Then $E$ induces a $k$-algebra homomorphism $$\varphi_E :T\to R[\varepsilon]/\varepsilon^{m+1}.$$ We have $\varphi_E(I)\subset\varepsilon^m\cdot R[\varepsilon]/\varepsilon^{m+1}$ and $ \varphi_E(I^2)=0$ so that $\varphi_E$ induces $$f_E:I/I^2\to \varepsilon^m\cdot R[\varepsilon]/\varepsilon^{m+1}\simeq R.$$
 We define $\operatorname{ob}_m(D)$ as the image of $f_E$ under $\operatorname{Hom}_R(I/I^2,R)\to T^1_{R/k}.$ 
\end{definition}

We note that $\operatorname{ob}_{mn}(D[n])=\operatorname{ob}_{m}(D)$ for another integer $n>1$ because $E_{m}=E[n]_{mn}$ in the above construction.

\begin{proposition}[{= \cite[Proposition 2.12]{miy2025}}]\label{obstruction independence} In the situation of Definition \ref{obstruction sp M,Ob}, the map $\operatorname{ob}_m$ is a group homomorphism independent of the choices of $T, I, E$. A HS-derivation $D\in \operatorname{HS}_k^{m-1}(R)$ extends to an element of $\operatorname{HS}_k^{m}(R)$ if and only if $\operatorname{ob}_m(D)=0$. For every $D\in \operatorname{HS}_k^{m-1}(R)$ and $x\in R$, we have $$\operatorname{ob}_m(x\bullet D)=x^m\cdot \operatorname{ob}_m(D).$$
\end{proposition}

\begin{definition}[{= \cite[Definition 2.13]{miy2025}}]\label{Ob}
  Let $R$ be a $k$-algebra essentially of finite type and $m\ge 2 $ an integer.  We define the $m$-th \emph{obstruction module} $\operatorname{Ob}^m_R\subset T^1_{R/k}$ as the image of $\operatorname{ob}_m$. We set $\operatorname{Ob}_R^{1}=0\subset T^1_{R/k}$. At this point, this is an abelian group.
\end{definition}

\begin{proposition}[{= \cite[Proposition 2.14]{miy2025}}]\label{Ob module}
In the situation of Definition \ref{obstruction sp M,Ob}, let us assume $m=p^i, i\ge1$. Then $\operatorname{Ob}^{p^i}_R$ is a finitely generated $R^{p^i}$-module.  If $S\subset R$ is a multiplicatively closed subset, then $(S^{p^i})^{-1}\operatorname{Ob}_R^{p^i}=\operatorname{Ob}^{p^i}_{S^{-1}R}$. 
\end{proposition}

\begin{proposition}[{= \cite[Proposition 2.16]{miy2025}}]\label{np=p}
In the situation of Definition \ref{obstruction sp M,Ob}, let us assume $m=p^in,i\ge0,n\ge2$ where $p\nmid n$. Then we have $\operatorname{Ob}^m_R=\operatorname{Ob}^{p^i}_R$.
\end{proposition}

By Propositions \ref{Ob module} and \ref{np=p}, we obtain:

\begin{theorem}[{= \cite[Theorem 2.17]{miy2025}}]\label{Ob sheaf}
    Let $X$ be a scheme essentially of finite type over $k$. For each $m\ge 1$, there exists a subsheaf
    $$\operatorname{Ob}_X^{m}\subset T_X^1$$
    such that for every affine open subscheme $\operatorname{Spec}R\subset X,$ we have
    $$\operatorname{Ob}_X^{m}(\operatorname{Spec}R)=\operatorname{Ob}_R^m\subset T_R^1.$$ It is a sub-$\mathcal{O}_X^{p^i}$-module of $T_X^1$, where $i$ is the largest integer such that $p^i$ divides $m$. It is also a coherent $\mathcal{O}_X$-module via $\mathcal{O}_X\to \mathcal{O}_X^{p^i}$.
    
\end{theorem}

\begin{theorem}[{= \cite[Theorem 2.25]{miy2025}}]\label{duality}
Let $X$ be a scheme essentially of finite type over $k$. Then there exist two filtrations of sheaves on $X$ such that
$$\operatorname{Der}_k(\mathcal{O}_X)\supset\operatorname{Der}_k^p(\mathcal{O}_X)\supset\operatorname{Der}_k^{p^2}(\mathcal{O}_X)\supset\dots$$
and that
$$0\subset \operatorname{Ob}^p_X\subset\operatorname{Ob}^{p^2}_X\subset\dots\subset T^1_{X/k}.$$ For every $i=1,2,\dots,$ we have
$$\operatorname{Der}_k^{p^{i-1}}(\mathcal{O}_X)/\operatorname{Der}_k^{p^i}(\mathcal{O}_X)\simeq \operatorname{Ob}_X^{p^i}/\operatorname{Ob}_X^{p^{i-1}}. $$
\end{theorem}

At this point, we do not know whether $0\subset \operatorname{Ob}^p_X\subset\operatorname{Ob}^{p^2}_X\subset\dots$ becomes stationary. This is because each term carries a distinct module structure, and we cannot apply the Noetherian property of $T_{X/k}^1$.

\begin{definition}\label{FL def}
    Let $R$ [resp.\ $X$] be a $k$-algebra essentially of finite type [resp.\ $k$-scheme essentially of finite type]. Let $i\ge 1$ be an integer. Then we say that $R$ [resp.\ $X$] satisfies $\FL{i}$ if we have
    $$\operatorname{Ob}_R^{p^i}=\operatorname{Ob}_R^{p^{i+1}}=\operatorname{Ob}_R^{p^{i+2}}=\dots $$
    [resp. $\operatorname{Ob}_X^{p^i}=\operatorname{Ob}_X^{p^{i+1}}=\dots $].
\end{definition}

    Note that this is equivalent to say that $$\operatorname{Der}_k^{p^i}(R)=\operatorname{Der}_k^{p^{i+1}}(R)=\operatorname{Der}_k^{p^{i+2}}(R)=\dots.$$

\section{On $a$-boundedness}

\begin{definition}
    Let $a=(a_n)_{n=1,2,\dots}$ be a sequence in $\mathbb{N}$. Then we say that $a$ is {\em subadditive} if $a_{n+m}\le a_n+a_m$.
\end{definition}

\begin{lemma}
    Let $a=(a_n)_{n=1,2,\dots}$ and $b=(b_n)_{n=1,2,\dots}$ be sequences in $\mathbb{N}$. Let $a+b$ and $\max\{a,b\}$ be sequences defined by
    $$a+b:=(a_n+b_n)_n,\ \ \max\{a,b\}:=(\max\{a_n,b_n\})_n.$$
    Then if $a$ and $b$ are subadditive, $a+b$ and $\max\{a,b\}$ are also subadditive.
\end{lemma}

\begin{definition}\label{a-bounded}
    Let $a=(a_n)_{n=1,2,\dots}$ be a subadditive sequence in $\mathbb{N}$, $R$ be a $k$-algebra and $\mathfrak{m}\subset R$ be an ideal. Let $D=(D_n)_{n=0}^m \in \HS_k ^m (R)$ for $1\le m\le\infty$. Then we say that $D$ is {\em $\mathfrak{m}$-adically $a$-bounded} [resp.\ {\em $\mathfrak{m}$-logarithmically $a$-bounded}] if
    $$D_n(R)\subset \mathfrak{m}^{a_n}$$
    for every $n$ [resp.\
    $$D_n(\mathfrak{m}^{n'})\subset \mathfrak{m}^{a_n+n'}$$
    for every $n,n'$]. When $\mathfrak{m}$ is clear from the context (e.g., when $R$ is a local ring and $\mathfrak{m}$ is the maximal ideal), we simply say {\em $a$-bounded} [resp.\ {\em logarithmically $a$-bounded}].
\end{definition}

Note that the composition of logarithmically $a$-bounded HS-derivations is again logarithmically $a$-bounded because $a$ is subadditive.

\begin{lemma}\label{a to log lem}
    In the setting of \ref{a-bounded}, let
    $$\widetilde{a}=(a_1+1,a_2+1,a_3+1,\dots).$$
    If $D=(D_n)_{n=0}^m \in \HS_k ^m (R)$ is $\widetilde{a}$-bounded, then it is logarithmically $a$-bounded.
\end{lemma}

\begin{proof}
    First observe that $\widetilde{a}$ is also subadditive. We show that by induction of $n'=0,1,\dots,$ we have
    $$D_n(\mathfrak{m}^{n'})\subset \mathfrak{m}^{a_n+n'}$$
    for every $n$. The case $n'=0$ is trivial as $\widetilde{a}\ge a$. Suppose that the claim is true for $n'$ and let $x\in \mathfrak{m}^{n'}, y\in \mathfrak{m}$. Then
    $$D_n(xy)=\sum_{i=0}^{n}D_{i}(x)D_{n-i}(y).$$
    For $i=0,1,\dots, n-1$, we have $$D_{i}(x)D_{n-i}(y)\in \mathfrak{m}^{a_{i}+n'}\cdot \mathfrak{m}^{a_{n-i}+1}\subset  \mathfrak{m}^{a_{n}+n'+1}.$$ For $i=n$, we also have $$D_n(x)D_{0}(y)\in \mathfrak{m}^{a_{n}+n'}\cdot \mathfrak{m}=\mathfrak{m}^{a_{n}+n'+1}.$$ Thus, we have $D_n(xy)\in \mathfrak{m}^{a_{n}+n'+1}$. As $\mathfrak{m}^{n'+1}$ is generated by such $xy$, this completes the induction.
    \end{proof}

\section{The local case}
In this section, we assume the following setting.

\begin{setting}\label{local setting}
    $(T,\mathfrak{m})$ is a formally smooth local $k$-algebra that is a localization of a polynomial ring in finite variables at a prime ideal (not necessarily maximal) and $I\subset \mathfrak{m}$ is an ideal of $T$. Set $(R,\mathfrak{n})=(T/I,\mathfrak{m}/I)$; this is a local $k$-algebra essentially of finite type. $X=\Spec R$ and $U=X\setminus \{\mathfrak{n}\}$ are $k$-schemes essentially of finite type and $\iota: U\hookrightarrow X$ is an open immersion. The $k$-scheme $U$ satisfies $\FL{\tau}$ for an integer $\tau\ge 1$. We put $q=p^{\tau}$.
\end{setting}

Note that this also includes the case where $X$ consists of a single point, where $U=\emptyset$ satisfies $\FL{\tau}$ for obvious reasons. The aim of this section is to show that $X$ itself satisfies $\FL{\theta}$ for another sufficiently large integer $\theta$. The idea is to start with a bounded HS-derivation and modify it step by step so that the obstruction at each stage becomes zero while the lower terms remain unchanged. To achieve this, we first obtain certain explicit positive constants and use them to construct suitable subadditive sequences.

\begin{lemma}\label{hom lem}
    Let $M:=I/I^2$ and let $H_n:=\Hom_R(M,\mathfrak{n}^n)$ for $n=1,2,\dots.$ We regard $H_n$ as submodules of $H:=\Hom_R(M,R)$. Then there exists a positive integer $C_1$ such that $$H_n\subset \mathfrak{n}^{n-C_1}\cdot H\quad (\forall n\ge C_1).$$
\end{lemma}

\begin{proof}
    Let $R^{e_1}\to R^{e_2}\to M\to 0$ be
    a finite presentation of $M$. For each $n=1,2,\dots$, we have the following commutative diagram
\[\begin{tikzcd}
	0 & {H_n} & {(\mathfrak{n}^n)^{e_2}} & {(\mathfrak{n}^n)^{e_1}} \\
	0 & H & {R^{e_2}} & {R^{e_1}}
	\arrow[from=1-1, to=1-2]
	\arrow[from=1-2, to=1-3]
	\arrow[hook, from=1-2, to=2-2]
	\arrow[from=1-3, to=1-4]
	\arrow[hook, from=1-3, to=2-3]
	\arrow[hook, from=1-4, to=2-4]
	\arrow[from=2-1, to=2-2]
	\arrow[from=2-2, to=2-3]
	\arrow[from=2-3, to=2-4]
\end{tikzcd}\]
From this diagram, we obtain $H_n=H\cap (\mathfrak{n}^n)^{e_2}$ where the intersection is taken inside $R^{e_2}$. Applying the Artin-Rees lemma to the inclusion $H\subset R^{e_2}$ with respect to the $\mathfrak{n}$-adic filtration yields the claim.
\end{proof}

\begin{lemma}\label{G lem}
    Let $M:=I/I^2$ and let $G$ be the image of the natural homomorphism $\Hom_R(\Omega_{T/k}\otimes_TR,R)\to \Hom_R(M,R)=H$. Then there exists a positive integer $C_2$ such that $$G\cap (\mathfrak{n}^n\cdot H)\subset \mathfrak{n}^{n-C_2}\cdot G\quad (\forall n\ge C_2).$$
\end{lemma}
\begin{proof}
    This also follows from the Artin-Rees lemma.
\end{proof}

\begin{definition}
    Let $i\ge 0$ be an integer, let $S$ be a $k$-algebra and let $J\subset S$ be an ideal. Define $$J^{[p^i]}:=\{x^{p^i}|x\in J\},$$ an ideal of $S^{p^i}$. It is the image of $J$ under the surjection $S\to S^{p^i}$. In particular, we have a natural inclusion $J^{[p^i]}S\subset J^{p^i}$ of ideals of $S.$ Also, observe that $(J^j)^{[p^i]}=(J^{[p^i]})^j$ for any $j\ge 0$.
\end{definition}

Note that there exists a positive integer $c$ such that for every $i\ge 0$ and $n\ge 0$, we have $(\mathfrak{n}^{[p^i]})^nR\supset \mathfrak{n}^{p^i(n+c)}$. Indeed, suppose that $\mathfrak{n}$ is generated by $x_1,\dots,x_m$. Then $\mathfrak{n}^{p^i(n+c)}$ is generated by monomials of the form
$$x_1^{\nu_1}\dots x_m^{\nu_m},\ \nu_1+\dots +\nu_m=p^i(n+c).$$
 For any $c\ge m-1$, we see that $x_1^{\nu_1}\dots x_m^{\nu_m}\in (\mathfrak{n}^{[p^i]})^nR$.

\begin{lemma}\label{i lem}
    There exists a positive integer $C_3$ such that, for every integer $i\ge \tau$, we have $$\operatorname{Ob}_R^{p^i}\cap (\mathfrak{n}^{nq}\cdot T_{R/k}^1)\subset (\mathfrak{n}^{n-C_3})^{[q]} \cdot \operatorname{Ob}_R^{q}\quad (\forall n\ge C_3).$$
\end{lemma}

\begin{proof}
Let $F$ be the subsheaf of $T_{X/k}^1$ consisting of sections $s$ such that $s|_U\in \operatorname{Ob}_U^{q}$. Then $F$ is a coherent subsheaf of $T_{X/k}^1$ via $\mathcal{O}_X\to \mathcal{O}_X^{q}$ (note that $R$ is finite over $R^q$ as $k$ is algebraically closed). Moreover, we have
$$\operatorname{Ob}_U^{q}=\operatorname{Ob}_U^{p^{\tau+1}}=\operatorname{Ob}_U^{p^{\tau+2}}=\dots=F|_U.$$
This implies that
$$\operatorname{Ob}_X^{q}\subset\operatorname{Ob}_X^{p^{\tau+1}}\subset\operatorname{Ob}_X^{p^{\tau+2}}\subset \dots\subset F$$
by \ref{duality}. Hence, it suffices to show the existence of $C_3$ such that $F\cap (\mathfrak{n}^{nq}\cdot T_{X/k}^1)\subset (\mathfrak{n}^{n-C_3})^{[q]} \cdot \operatorname{Ob}_X^{q}$ for $n\ge C_3$. Choose $c$ as above so that $(\mathfrak{n}^{[q]})^{n-c}R\supset \mathfrak{n}^{nq}$ for every $n\ge c$.
By the Artin-Rees lemma, there exists $c_1\ge 0$ such that
$$F\cap ((\mathfrak{n}^{[q]})^n\cdot T_{X/k}^1)\subset (\mathfrak{n}^{[q]})^{n-c_1}\cdot F$$
if $n\ge c_1$. As the quotient sheaf $F/\operatorname{Ob}_X^q$ is supported at $\{\mathfrak{n}\}$, there exists $c_2\ge 0$ such that 
$$(\mathfrak{n}^{[q]})^{c_2}\cdot F\subset \operatorname{Ob}_X^q.$$
We set $C_3:=c+c_1+c_2$. Then, if $n\ge C_3$, we have
$$F\cap (\mathfrak{n}^{nq}\cdot T_{X/k}^1)\subset F\cap ((\mathfrak{n}^{[q]})^{n-c}\cdot T_{X/k}^1)\subset (\mathfrak{n}^{[q]})^{n-c-c_1}\cdot F\subset (\mathfrak{n}^{[q]})^{n-c-c_1-c_2} \cdot \operatorname{Ob}_X^{q}$$
and this completes the proof.
\end{proof}

\begin{lemma}\label{j lem}
    There exists a positive integer $C_4$ such that, for every $i=0,1,\dots,\tau-1$, we have $$\operatorname{Ob}_R^{p^i}\cap (\mathfrak{n}^{np^i}\cdot T_{R/k}^1)\subset (\mathfrak{n}^{n-C_4})^{[p^i]} \cdot \operatorname{Ob}_R^{p^i}\quad (\forall n\ge C_4).$$
\end{lemma}

\begin{proof}
    It suffices to show the claim when $i$ is fixed. Take $c_3$ so that $(\mathfrak{n}^{[p^i]})^{n-c_3}R\supset \mathfrak{n}^{p^in}$ for every $n\ge c_3$. Applying the Artin-Rees lemma yields $c_4\ge 0$ such that
$$\operatorname{Ob}_R^{p^i}\cap ((\mathfrak{n}^{[p^i]})^n\cdot T_{R/k}^1)\subset (\mathfrak{n}^{[p^i]})^{n-c_4}\cdot \operatorname{Ob}_R^{p^i}$$
if $n\ge c_4$. Then we can take $C_4:=c_3+c_4$.
\end{proof}

\begin{definition}\label{seq def}
    Let $e_1, n_1,q_1$ be positive integers. Define three sequences of positive integers
    \begin{equation*}
         \left\{ \begin{alignedat}{2}a[e_1,n_1,q_1] &=(a[e_1,n_1,q_1]_n)_{n=1,2,\dots},\\
    \widetilde{a}[e_1,n_1,q_1] &=(\widetilde{a}[e_1,n_1,q_1]_n)_{n=1,2,\dots},\\
    b[e_1,n_1,q_1] &=(b[e_1,n_1,q_1]_n)_{n=1,2,\dots},\end{alignedat}\right .
    \end{equation*}
    by
    \begin{equation*}
   \left\{ \begin{alignedat}{2}a[e_1,n_1,q_1]_n& =e_1(q_1+1)\cdot \lceil \frac{nq_1}{n_1} \rceil,\\
    \widetilde{a}[e_1,n_1,q_1]_n& =e_1\cdot (\lfloor \frac{nq_1(q_1+1)}{n_1} \rfloor+1),\\
    b[e_1,n_1,q_1] &=\max\{a[e_1,n_1,q_1],\ \widetilde{a}[e_1,n_1,q_1]\}.\end{alignedat}\right .
    \end{equation*}
    Here, for a real number $x$, $\lceil x \rceil$ [resp.\ $\lfloor x \rfloor$] denotes the smallest integer greater than or equal to $x$ [resp. the greatest integer smaller than or equal to $x$].
    
\end{definition}

\begin{lemma}
    In the setting of \ref{seq def}, we have
    $$a[e_1,n_1,q_1]\ge a[e_1,n_1+1,q_1],$$
    $$a[e_1,n_1,q_1]\ge \widetilde{a}[e_1,n_1+1,q_1].$$
    In particular, $a[e_1,n_1,q_1]\ge b[e_1,n_1+1,q_1]\ge a[e_1,n_1+1,q_1]$.
\end{lemma}

\begin{proof} The first inequality is clear.
    For every $n\ge 1$, we have
    $$(q_1+1)\lceil \frac{nq_1}{n_1} \rceil\ge \lceil \frac{nq_1(q_1+1)}{n_1} \rceil> \lfloor \frac{nq_1(q_1+1)}{n_1+1} \rfloor.$$
    Note that the last inequality is strict. Thus, we have $a[e_1,n_1,q_1]\ge \widetilde{a}[e_1,n_1+1,q_1].$ The last inequalities follow from the definition of $b[e_1,n_1,q_1]$.
\end{proof}

\begin{lemma}
    For any positive integers $e_1, n_1,q_1$, the sequences $a[e_1,n_1,q_1],$ $b[e_1,n_1,q_1]$ are subadditive.
\end{lemma}

\begin{definition}
    Let $D\in \HS_k^\infty(T)$ and $n\in \mathbb{N}$. Then we say that $D$ is {\em $I$-logarithmic up to $n$} if 
    $D_i(I)\subset I$ for $i=0,\dots, n$.
\end{definition}

\begin{definition}\label{l_n def}
    Define a sequence $l_1,l_2,\dots$ of positive integers by
    $$l_n=n/q$$
    if $n$ is divisible by $q$ and
    $$l_n=n'$$
    if $n$ factorizes as $n=n'p^i$ with $p\nmid n',\ i<\tau$. Note that $l_n$ tends to $\infty$ as $n\to \infty$ because $l_n\ge n/q$.
\end{definition}

\begin{proposition}\label{a to b extension prop}
    Let $e$ be an integer with
    $$e\ge C_1+C_3\cdot q+C_4\cdot q+q$$
    where $C_1,C_3,C_4$ are as in \ref{hom lem}-\ref{j lem}. Let $m$ be a positive integer and $D\in \HS_k^\infty(T)$ be such that:
    \begin{enumerate}
        \item $D$ is $I$-logarithmic up to $m$, thus induces $\widetilde{D}\in \HS_k^m(R),$
        \item $D$ is logarithmically $a[e,m,q]$-bounded.
    \end{enumerate}
    Then there exists $E\in \HS_k^\infty(T)$ such that:
     \begin{enumerate}
        \item $E$ is $I$-logarithmic up to $m$, thus induces $\widetilde{E}\in \HS_k^m(R),$
        \item $\operatorname{ob}_{m+1}(\widetilde{E})=0$,
        \item $E$ is logarithmically $b[e,m+1,q]$-bounded,
        \item $D$ and $E$ coincide up to $l_{m+1}-1$, i.e., $D_1=E_1,\dots,D_{l_{m+1}-1}=E_{l_{m+1}-1}$.
    \end{enumerate}
\end{proposition}

\begin{proof}
    Note that the composition
    $$T\overset{D_{m+1}}{\longrightarrow} T\twoheadrightarrow R$$
    restricts to an $R$-homomorphism
    $$f: I/I^2\to R$$
    as $D$ is $I$-logarithmic up to $m$. By Definition \ref{obstruction sp M,Ob},  the image of $f$ in $T^1_{R/k}$ is equal to $\operatorname{ob}_{m+1}(\widetilde{D})$. As $D$ is $a[e,m,q]$-bounded, $f$ factors as $I/I^2\to \mathfrak{n}^{a[e,m,q]_{m+1}}\hookrightarrow R.$ Note that $ a[e,m,q]_{m+1} \ge eq(q+1)+e(q+1)$. 
     By \ref{hom lem}, we have $$f\in \mathfrak{n}^{eq(q+1)+e(q+1)-C_1}\cdot \Hom_R(I/I^2,R).$$ In particular, $$f\in \mathfrak{n}^{eq(q+1)+eq+C_3\cdot q+C_4\cdot q+q}\cdot \Hom_R(I/I^2,R).$$
    This implies
    $$\operatorname{ob}_{m+1}(\widetilde{D})\in \mathfrak{n}^{eq(q+1)+eq+C_3\cdot q+C_4\cdot q+q}\cdot T^1_{R/k}.$$
    Let us write $m+1=m'p^i,\ p\nmid m'$. From this point on, we treat the case $i\ge \tau$ and the case $i<\tau$ separately.\\
    First, assume that $i\ge \tau$. 
    We can apply \ref{i lem} as $\operatorname{Ob}_R^{m+1}=\operatorname{Ob}_R^{p^i}$ by \ref{np=p}. We have
    $$\operatorname{ob}_{m+1}(\widetilde{D})\in (\mathfrak{n}^{e(q+1)+e+C_4 +1})^{[q]}\cdot  \operatorname{Ob}_R^{q}\subset (\mathfrak{n}^{e(q+1)+e+1})^{[q]}\cdot  \operatorname{Ob}_R^{q}.$$
    This means that there exist $\widetilde{D}^1,\dots,\widetilde{D}^n\in \HS_k^{q-1}(R)$ and $x_1,\dots,x_n\in \mathfrak{n}^{e(q+1)+e+1}$ such that 
    $$\operatorname{ob}_{m+1}(\widetilde{D})+x_1^q\cdot\operatorname{ob}_q(\widetilde{D}^1)+\dots+x_n^q\cdot \operatorname{ob}_q(\widetilde{D}^n)=0.$$
    The latter is equivalent to
    $$\operatorname{ob}_{m+1}(\widetilde{D})+\operatorname{ob}_{q}(x_1\bullet\widetilde{D}^1)+\dots + \operatorname{ob}_{q}(x_n\bullet\widetilde{D}^n)=0$$
    by \ref{obstruction independence}. We can lift $\widetilde{D}^1,\dots,\widetilde{D}^n\in \HS_k^{q-1}(R)$ to $$\widetilde{E}^1,\dots,\widetilde{E}^n\in \HS_k^\infty(T)$$ by formal smoothness of $T$ and $x_1,\dots,x_n\in \mathfrak{n}^{e(q+1)+e+1}$ to $$y_1,\dots,y_n\in \mathfrak{m}^{e(q+1)+e+1}.$$
    Note that, for every $j=1,\dots,n$, $F^j:=(y_j\bullet \widetilde{E}^j)[m'p^{i-\tau}]$ is logarithmically $b[e,m+1,q]$-bounded: To prove this, it suffices to consider $F^j_{sm'p^{i-\tau}}$ for each $s=1,2,\dots$ because other terms of $F^j$ are zero. The image of $F^j_{sm'p^{i-\tau}}$ is contained in $\mathfrak{m}^{se(q+1)+se+s}$. On the other hand,
    \begin{align*}
   b[e,m+1,q]_{sm'p^{i-\tau}}& =\max\{e(q+1)\cdot \lceil \frac{sm'p^{i-\tau}q}{m+1} \rceil, e\cdot (\lfloor \frac{sm'p^{i-\tau}q(q+1)}{m+1} \rfloor+1)\}\\
   &=\max\{ e(q+1)\cdot s,e\cdot (s(q+1)+1)\}\\
   &=es(q+1)+e\\
   &<se(q+1)+se+s.
    \end{align*}
    By \ref{a to log lem}, we see that $F^j$ is logarithmically $b[e,m+1,q]$-bounded. Then
    $$E:=D\circ F^1\circ\dots\circ F^n$$
    satisfies the required properties: We have just shown that $E$ is logarithmically $b[e,m+1,q]$-bounded. As each $\widetilde{E}^j$ is $I$-logarithmic up to $q-1$, each $F^j$ is $I$-logarithmic up to $(q-1)m'p^{i-\tau}+m'p^{i-\tau}-1=m$. Hence, also $E$ is $I$-logarithmic up to $m$ and induces
    $$\widetilde{E}=\widetilde{D}\circ (x_1\bullet \widetilde{D}^1)[m'p^{i-\tau}]\circ\dots \circ (x_n\bullet \widetilde{D}^n)[m'p^{i-\tau}] \in \HS^m_k(R).$$
    From this expression, we obtain $\operatorname{ob}_m(\widetilde{E})=0$. Finally, note that for each $j$, the first $m'p^{i-\tau}-1$ terms of $F^j$ are zero
    and note that $l_{m+1}=m'p^{i-\tau}$. This implies that $D$ and $E$ coincide up to $l_{m+1}-1$. This completes the proof in the case $i\ge \tau$.\\
    Next, suppose that $i<\tau $.
    As $$\operatorname{ob}_{m+1}(\widetilde{D})\in \mathfrak{n}^{eq(q+1)+eq+C_3\cdot q+C_4\cdot q+q}\cdot T^1_{R/k},$$ we have in particular
    $$\operatorname{ob}_{m+1}(\widetilde{D})\in \mathfrak{n}^{(e(q+1)p^{\tau-i}+ep^{\tau-i}+C_3\cdot p^{\tau-i}+C_4\cdot p^{\tau-i}+p^{\tau-i})p^i}\cdot T^1_{R/k}.$$
    We can apply \ref{j lem} as $\operatorname{Ob}_R^{m+1}=\operatorname{Ob}_R^{p^i}$ by \ref{np=p}. We have
    $$\operatorname{ob}_{m+1}(\widetilde{D})\in (\mathfrak{n}^{e(q+1)p^{\tau-i}+ep^{\tau-i}+C_3\cdot p^{\tau-i}+C_4\cdot (p^{\tau-i}-1)+p^{\tau-i}})^{[p^i]}\cdot  \operatorname{Ob}_R^{p^i}$$
    and in particular
    $$\operatorname{ob}_{m+1}(\widetilde{D})\in (\mathfrak{n}^{e(q+1)p^{\tau-i}+e+1})^{[p^i]}\cdot  \operatorname{Ob}_R^{p^i}.$$
    Then there exist $\widetilde{D}^1,\dots,\widetilde{D}^n\in \HS_k^{p^i-1}(R)$ and $x_1,\dots,x_n\in \mathfrak{n}^{e(q+1)p^{\tau-i}+e+1}$ such that 
    $$\operatorname{ob}_{m+1}(\widetilde{D})+\operatorname{ob}_{p^i}(x_1\bullet\widetilde{D}^1)+\dots + \operatorname{ob}_{p^i}(x_n\bullet\widetilde{D}^n)=0.$$ We can lift $\widetilde{D}^1,\dots,\widetilde{D}^n\in \HS_k^{p^i-1}(R)$ to $$\widetilde{E}^1,\dots,\widetilde{E}^n\in \HS_k^\infty(T)$$ by formal smoothness of $T$ and $x_1,\dots,x_n\in \mathfrak{n}^{e(q+1)p^{\tau-i}+e+1}$ to $$y_1,\dots,y_n\in \mathfrak{m}^{e(q+1)p^{\tau-i}+e+1}.$$
    Note that, for every $j=1,\dots,n$, $F^j:=(y_j\bullet \widetilde{E}^j)[m']$ is logarithmically $b[e,m+1,q]$-bounded: For each $s=1,2,\dots,$ the image of $F^j_{sm'}$ is contained in $\mathfrak{m}^{se(q+1)p^{\tau-i}+se+s}$. On the other hand,
    \begin{align*}
   b[e,m+1,q]_{sm'}& =\max\{ e(q+1)\cdot \lceil \frac{sm'q}{m+1} \rceil,e\cdot (\lfloor \frac{sm'q(q+1)}{m+1} \rfloor+1)\}\\
   &=\max\{ e(q+1)p^{\tau-i},e\cdot (s(q+1)p^{\tau-i}+1)\}\\
   &=es(q+1)p^{\tau-i}+e\\
   &<es(q+1)p^{\tau-i}+es+s.
    \end{align*}
    By \ref{a to log lem}, we see that $F^r$ is logarithmically $b[e,m+1,q]$-bounded. Then
    $$E:=D\circ F^1\circ\dots\circ F^n$$
    satisfies the required properties for the same reasons as before. This completes the proof in the case $i<\tau$.
\end{proof}

\begin{proposition}\label{b to a extension prop}
    Let $e$ be an integer with
    $$e\ge C_1+C_2+1$$
    where $C_1,C_2$ are as in \ref{hom lem}, \ref{G lem}. Let $m$ be a positive integer and $E\in \HS_k^\infty(T)$ such that:
         \begin{enumerate}
        \item $E$ is $I$-logarithmic up to $m$, thus induces $\widetilde{E}\in \HS_k^m(R),$
        \item $\operatorname{ob}_{m+1}(\widetilde{E})=0$,
        \item $E$ is logarithmically $b[e,m+1,q]$-bounded.
    \end{enumerate}
    Then there exists $D^+\in \HS_k^\infty (T)$ such that:
        \begin{enumerate}
        \item $D^+$ is $I$-logarithmic up to $m+1$,
        \item $D^+$ is logarithmically $a[e,m+1,q]$-bounded,
        \item $D^+$ and $E$ coincide up to $m$, i.e., $D^+_1=E_1,\dots,D^+_{m}=E_{m}$.
    \end{enumerate}
\end{proposition}

\begin{proof}
     Note that the composition
    $$T\overset{E_{m+1}}{\longrightarrow} T\to R$$
    restricts to an $R$-homomorphism
    $$f: I/I^2\to R.$$
    By definition, the image of $f$ in $T^1_{R/k}$ is equal to $\operatorname{ob}_{m+1}(\widetilde{E})$, which is $0$ by assumption. Thus, $f$ is in the image $G$ of
    $$\Hom_R(\Omega_{T/k}\otimes_T R,R)\to \Hom_R (I/I^2,R).$$
    As $E$ is $b[e,m+1,q]$-bounded, $f$ factors as $I/I^2\to \mathfrak{n}^{b[e,m+1,q]_{m+1}}\hookrightarrow R.$ Note that 
    $$ b[e,m+1,q]_{m+1} =\max\{e(q+1)\cdot \lceil \frac{(m+1)q}{m+1} \rceil,  e\cdot (\lfloor \frac{(m+1)q(q+1)}{m+1} \rfloor+1) \}=eq(q+1)+e.$$
    By \ref{hom lem}, we have
    $$f\in \mathfrak{n}^{eq(q+1)+e - C_1}\cdot \Hom_R (I/I^2,R).$$
    By \ref{G lem}, we have
    $$f\in \mathfrak{n}^{eq(q+1)+e - C_1-C_2}\cdot G\subset \mathfrak{n}^{eq(q+1)+1}G.$$
    This means that we can take $g\in \mathfrak{n}^{eq(q+1)+1}\Hom_R(\Omega_{T/k}\otimes_T R,R) $ that restricts to $f$. As $T$ is formally smooth over $k$, we can take $x_1,\dots,x_n\in \mathfrak{m}^{eq(q+1)+1}$ and $d_1,\dots,d_n\in \operatorname{Der}_k(T,T)$ such that
    $x_1d_1+\dots +x_nd_n$
    restricts to $-g$. Again as $T$ is formally smooth over $k$, $(\operatorname{id}, d_j)\in \HS_k^1(T)$ lifts to $\widetilde{E}^j\in \HS_k^\infty (T)$ for every $j$. Set $F^j:=(x_j\bullet \widetilde{E}^j)[m+1]$. Then $F^j$ is logarithmically $a[e,m+1,q]$-bounded: For each $s=1,2,\dots,$ the image of $F^j_{s(m+1)}$ is contained in $\mathfrak{m}^{seq(q+1)+s}$. On the other hand,
    $$a[e,m+1,q]_{s(m+1)}=e(q+1)\cdot \lceil \frac{s(m+1)q}{m+1} \rceil=seq(q+1)< seq(q+1)+s.$$
    By \ref{a to log lem}, we see that $F^j$ is logarithmically $a[e,m+1,q]$-bounded. Then
    $$D^+=E\circ F^1\circ \dots \circ F^n$$
    satisfy the required properties: (ii) and (iii) follow from the construction. For (i), we note that
    $$D^+_{m+1}=E_{m+1}+x_1d_1+\dots+x_nd_n.$$
    This induces $f-f=0\in \operatorname{Hom}_R(I/I^2,R)$ so that $D^+$ is $I$-logarithmic up to $m+1$.
    This completes the proof.
    \end{proof}

\begin{proposition}
    Let $e$ be an integer with
    $$  e\ge C_1+ C_2 + C_3\cdot q+C_4\cdot q+q$$
    where the $C_i$ are as in \ref{hom lem}-\ref{j lem}. Let $m$ be a positive integer and $D\in \HS_k^\infty(T)$ such that:
    \begin{enumerate}
        \item $D$ is $I$-logarithmic up to $m$,
        \item $D$ is logarithmically $a[e,m,q]$-bounded.
    \end{enumerate}    
    Then there exists $D^+\in \HS_k^\infty (T)$ such that:
    \begin{enumerate}
        \item $D^+$ is $I$-logarithmic up to $m+1$,
        \item $D^+$ is logarithmically $a[e,m+1,q]$-bounded,
        \item $D^+$ and $D$ coincide up to $l_{m+1}-1$, i.e., $D^+_1=D_1,\dots,D^+_{l_{m+1}-1}=D_{l_{m+1}-1}$.
        \end{enumerate}
\end{proposition}

\begin{proof}
    This is a consequence of Proposition \ref{a to b extension prop}, \ref{b to a extension prop}.
\end{proof}

\begin{theorem}\label{eq finie thm}
        Let $e$ be an integer with
    $$  e\ge C_1+ C_2 + C_3\cdot q+C_4\cdot q+q$$ where the $C_i$ are as in \ref{hom lem}-\ref{j lem}. Let $\operatorname{Der}_k^{q}(R)$ be the module of $q$-integrable derivations of $R$. Then any element of
    $\mathfrak{n}^{e(q+1)+1} \cdot \operatorname{Der}_k^{q}(R)$
    is $m$-integrable for any positive integer $m$. In particular, the set of leaps produced by $\operatorname{Der}_k^{q}(R)$ is finite.
\end{theorem}

\begin{proof}
    Suppose that $d\in \operatorname{Der}_k^{q}(R)$ and  $x\in \mathfrak{n}^{e(q+1)+1} $. By assumption, there exists $\widetilde{E}\in \HS_k^q(R)$ such that $\widetilde{E}_1=d$. We can lift $\widetilde{E}$ to $E\in \HS_k^\infty(T)$ by formal smoothness of $T$. Let $y\in \mathfrak{m}^{e(q+1)+1}$ be a preimage of $x$. Then $D:=y \bullet E$ is $I$-logarithmic up to $q$ and logarithmically $a[e,q,q]$-bounded (note that $a[e,q,q]=(e(q+1),2e(q+1),3e(q+1),\dots)$). Let $m$ be an arbitrary positive integer. Applying the above proposition recursively, we obtain $D^+\in \HS_k^\infty (T)$ such that $D^+$ is $I$-logarithmic up to $m$ and $D$ and $D^+$ coincide up to $\min\{l_{q+1}-1,l_{q+2}-1,\dots, l_{m}-1\}.$ Note that $l_{n}\ge 2$ for any $n>q$. In particular, $y\cdot E_1=D^+_1$. This means that $x\cdot d$ is $m$-integrable. This shows the first part of the theorem. The second part is a consequence of the following theorem.
    \end{proof}

The next theorem is a slight variant of \cite[Theorem 3.1]{bravo2024finitenessleapssensehasseschmidt}, \cite[Proposition 1.11]{NARVAEZMACARRO2024109441} and its proof is essentially the same. For the convenience of the reader, we include the proof here as well.

    \begin{theorem}\label{Her thm}
        Let $M$ be a sub-$R$-module of $\operatorname{Der}_k(R,R)$ and $\mathfrak{l}\subset \mathfrak{n}$ and ideal such that $R/\mathfrak{l}$ is of finite length. Assume that any $d\in \mathfrak{l}\cdot M$ is $m$-integrable for any positive integer $m$. Let $\operatorname{Leaps}(M)\subset \mathbb{N}$ be the set of leaps produced by $M$. Then $$\# \operatorname{Leaps}(M)\le \operatorname{length}_R(M/ \mathfrak{l}M).$$ In particular, $\operatorname{Leaps}(M)$ is finite.
    \end{theorem}

    \begin{proof}
        As $k$ is perfect, there exists a subfield $K\subset R/\mathfrak{l}$ containing $k$ such that $K\simeq R/\mathfrak{n}$. The length of $M/ \mathfrak{l} M$ is equal to the $K$-linear dimension of $M/ \mathfrak{l}M$.

        Now, let $L:=\operatorname{Leaps}(M)$. For each $s\in L$, we take $d_s\in M$ leaping at $s$. We claim that the set $\{\overline{d}_s\}_{s\in L}$ is $K$-linearly independent in $M/ \mathfrak{l}M$, where $\overline{d}_s$ is the image of $d_s$ in $M/ \mathfrak{l} M$.
        Suppose that there exists a relation such that
        $$a_1\overline{d}_{s_1}+\dots +a_n\overline{d}_{s_n}=0,\ a_i\in K$$
        where $a_1\neq 0$ and $s_1<\dots < s_n$. Then we have a relation in $M$ such that
        $$b_1d_{s_1}=-b_2d_{s_2}-\dots -b_nd_{s_n}+d,\ b_i\in R$$
        where $b_i$ are lifts of $a_i$ to $R$ and where $d\in \mathfrak{l}M$. Note that $b_1$ is invertible, so that 
        $$d_{s_1}=b_1^{-1}(-b_2d_{s_2}-\dots -b_nd_{s_n}+d).$$
        By assumption, there exists $D_i\in \HS_k^{s_i}(R)$ whose first component is equal to $d_{s_i}$ for each $i$. In addition, there exists $D\in \HS_k^{t}(R)$ whose first component is equal to $d$ with $t\ge s_2$. Then the first component of
        $$b_1^{-1}\bullet (((-b_2)\bullet D_2)\circ \dots \circ ((-b_n)\bullet D_n)\circ D)\in \HS_k^{s_2}(R)$$
        is equal to $d_{s_1}$. In the above expression, $D_i$ and $D$ are understood to be truncated to  length $s_2$. This implies that $d_{s_1}$ is at least $s_2$-integrable, which contradicts the choice of $d_{s_1}$. This proves the claim, and the theorem follows from it.
    \end{proof}

\begin{theorem}\label{sec 4 main thm}
    The set of leaps of $\operatorname{Der}_k(R,R)$ is finite. In particular, $R$ satisfies $\FL{\theta}$ for a sufficiently large integer $\theta$.
\end{theorem}

\begin{proof}
    The set of leaps of $\operatorname{Der}_k(R,R)$ is contained in
    $$\{p,p^2,\dots,q\}\cup \operatorname{Leaps}(\operatorname{Der}_k^q(R)),$$
    which is finite by Theorem \ref{eq finie thm}.
\end{proof}

\section{The global case}
In the previous section, we proved that if a local $k$-algebra $(R,\mathfrak{n})$ essentially of finite type satisfies the finiteness of leaps except the closed point, then so does $R$ itself (Setting \ref{local setting}, Theorem \ref{sec 4 main thm}). In this section, we apply this result to prove the finiteness of leaps for arbitrary $k$-schemes essentially of finite type.

\begin{theorem}\label{main thm global}
    Let $X$ be a $k$-scheme essentially of finite type. Then $X$ satisfies $\FL{\tau}$ for a sufficiently large integer $\tau$. In particular, the set of leaps of $X$ is finite.
\end{theorem}

\begin{proof}
    We construct inductively a sequence of pairs $(F_1,i_1),(F_2,i_2),\dots$ such that:
    \begin{enumerate}
        \item $i_1\le i_2\le i_3\le \dots$ are nonnegative integers,
        \item $F_n$ is a sub-$\mathcal{O}_X^{p^{i_n}}$-sheaf of $T_{X/k}^1$ that is coherent via $\mathcal{O}_X\to\mathcal{O}_X^{p^{i_n}}$,
        \item $\operatorname{Ob}_X^{p}\subset\operatorname{Ob}_X^{p^2}\subset \dots \subset F_n$ for every $n$,
        \item $F_1\supset F_2\supset F_3\supset \dots$,
        \item for every $n$, $\operatorname{Supp}(F_n/\operatorname{Ob}_X^{p^{i_n}})\supsetneq \operatorname{Supp}(F_{n+1}/\operatorname{Ob}_X^{p^{i_{n+1}}})$.
    \end{enumerate}
    By (v), this sequence must terminate at some $(F_\tau,i_\tau)$ where we have $$\operatorname{Supp}(F_\tau/\operatorname{Ob}_X^{p^{i_\tau}})=\emptyset,$$ i.e., $F_\tau=\operatorname{Ob}_X^{p^{i_\tau}}$. By (iii), $X$ satisfies $\FL{\tau}$ and in particular the set of leaps of $X$ is finite.

    We set $F_1=T_{X/k}^1,\ i_1=0$. Suppose that we have constructed $(F_n, i_n)$ and we would like to construct $(F_{n+1},i_{n+1})$ satisfying (i)-(v). Let us take a generic point $\xi$ of $\operatorname{Supp}(F_n/\operatorname{Ob}_X^{p^{i_n}})$ (i.e. a point without generalizations). Let $R$ be the local ring of $X$ at $\xi$ and let $Y=\operatorname{Spec}R,\ U:=Y\setminus \{\xi\}$. Then $U$ satisfies $\FL{i_n}$ as $\operatorname{Supp}(F_n|_U/\operatorname{Ob}_U^{p^{i_n}})=\emptyset$. Thus, we can apply the results of the previous section to the local $k$-algebra $R$. We can find an integer $i_{n+1}\ge i_{n}$ such that $Y$ satisfies $\FL{i_{n+1}}$. We define $F_{n+1}$ as the subsheaf of $F_n$ consisting of sections $s\in F_n$ such that $$s|_Y\in \operatorname{Ob}_Y^{p^{i_{n+1}}}.$$
    Conditions (i) and (iv) follow from the construction. We have the following cartesian diagram of $\mathcal{O}_X^{p^{i_{n+1}}}$-sheaves
\[\begin{tikzcd}
	{F_{n+1}} & {\iota_*\operatorname{Ob}_Y^{i_{n+1}}} \\
	{F_n} & {\iota_*\iota^*F_n}
	\arrow[from=1-1, to=1-2]
	\arrow[hook, from=1-1, to=2-1]
	\arrow[hook, from=1-2, to=2-2]
	\arrow[from=2-1, to=2-2]
\end{tikzcd}\]
    where $\iota: Y\to X$. In other words, we have $F_{n+1}=F_n\times_{\iota_*\iota^*F_n}\iota_*\operatorname{Ob}_Y^{i_{n+1}}$. This implies the quasi-coherence, and hence the coherence of $F_{n+1}$ in the sense stated in (ii). Condition (iii) follows from the fact that $Y$ satisfies $\FL{i_{n+1}}$ and from the construction. The restriction of the above diagram to $Y$ is
\[\begin{tikzcd}
	{\iota^*F_{n+1}} & {\iota^*\iota_*\operatorname{Ob}_Y^{i_{n+1}}=\operatorname{Ob}_Y^{i_{n+1}}} \\
	{\iota^*F_n} & {\iota^*\iota_*\iota^*F_n}
	\arrow[from=1-1, to=1-2]
	\arrow[hook, from=1-1, to=2-1]
	\arrow[hook, from=1-2, to=2-2]
	\arrow[equals, from=2-1, to=2-2],
\end{tikzcd}\]
    which is also cartesian. This implies $F_{n+1}|_Y=\iota^*F_{n+1}=\operatorname{Ob}_Y^{i_{n+1}}$ and $\xi \notin \operatorname{Supp}(F_{n+1}/\operatorname{Ob}_X^{p^{i_{n+1}}})$. In particular, (v) follows.
\end{proof}

As an affine case of the above theorem, we obtain the following.

\begin{theorem}\label{eft alg thm}
    Any $k$-algebra $R$ essentially of finite type satisfies $\FL{\tau}$ for a sufficiently large integer $\tau$. In particular, the set of leaps of $R$ is finite.
\end{theorem}

We now study $\infty$-integrability based on the finiteness of leaps. We begin with the following lemma.

\begin{lemma}\label{mn int lem}
    Let $R$ be a $k$-algebra essentially of finite type, let $m, \tau\ge 1$ be integers, and let $q:=p^\tau$. Suppose that $R$ satisfies $\FL{\tau}$. Let $D\in \HS_k^m(R)$. Let $l_1,l_2,\dots$ be the sequence of integers defined as in \ref{l_n def}. Then there exists $D^+\in \HS_k^{m+1}(R)$ such that $D$ and $D^+$ coincide up to $l_{m+1}-1$, i.e., $D_1=D^+_1,\dots, D_{l_{m+1}-1}= D_{l_{m+1}-1}^+$.
\end{lemma}

\begin{proof}
    Let $r:={m+1}/l_{m+1}$. As $\operatorname{Ob}_R^{p^\tau}=\operatorname{Ob}_R^{p^{\tau+1}}=\dots$, there exists $E\in \HS_k^{r-1}(R)$ such that $\operatorname{ob}_{m+1}(D)=-\operatorname{ob}_r(E)$. Then $D\circ E[l_{m+1}]$ extends to $D^+\in  \HS_k^{m+1}(R)$, whose first $l_{m+1}-1$ terms coincide with those of $D$.
\end{proof}

\begin{theorem}\label{mn int theorem}
    Let $R$ be a $k$-algebra essentially of finite type, let $n\ge 1$ be an integer, and let $D\in \HS_k^n(R)$. Then $D$ is $\infty$-integrable if and only if $D$ is $m$-integrable for every integer $m\ge n$. Here, we say that $D$ is $m$-integrable if $D$ is in the image of $\HS_k^m(R)\to \HS_k^n(R)$.
\end{theorem}

\begin{proof}
    The ``only if'' direction is clear. Suppose that $D$ is $m$-integrable for every integer $m\ge n$. Take a sufficiently large integer $N\gg n$ such that $l_{m+1}-1\ge n$ for every $m\ge N$. By assumption, there exists $E\in \HS_k^N(R)$ that restricts to $D$. Inductively, we define
    $$E^0=E\in \HS_k^N(R),$$
    $$E^{m+1}=(E^m)^+ \in \HS_k^{N+m+1}(R)$$
    for $m=0,1,\dots$, where $(E^m)^+$ is defined using the previous lemma. For fixed $m'=0,1,\dots$, we have a sequence of $k$-linear maps
    $$ E^{m}_{m'}, \ E^{m+1}_{m'},\ E^{m+2}_{m'},\dots$$
    for sufficiently large $m\gg 0$. Eventually it becomes stationary because $l_m$ tends to $\infty$ as $m\to \infty$. Thus, 
    $$E^\infty:=(\lim_{m\to \infty}E^m_{m'})_{m'=0}^\infty \in \HS_k^\infty(R)$$
    is a well-defined HS-derivation. Note that $E^\infty$ restricts to $D$ by the choice of $N$. Thus, $D$ is $\infty$-integrable.
\end{proof}

Setting $n=1$ in the above theorem, we obtain the following.
\begin{theorem}\label{infty integrable thm}
     Let $R$ be a $k$-algebra essentially of finite type. Then there exists $M>1$ such that
     $$ \operatorname{Der}_k^M(R)= \operatorname{Der}_k^{M+1}(R)=\dots= \operatorname{Der}_k^\infty(R)=\operatorname{Ider}_k(R).$$
\end{theorem}

\begin{corollary}\label{localization ider cor}
 Let $R$ be a $k$-algebra essentially of finite type. Then $\operatorname{Der}_k^\infty(R)=\operatorname{Ider}_k(R)$ is compatible with localization of $R$. That is, we have $$\operatorname{Ider}_k(S^{-1}R)=S^{-1}\operatorname{Ider}_k(R)$$
as $S^{-1}R$-modules for any multiplicatively closed subset $S\subset R$.
\end{corollary}
\begin{proof}
    This is a consequence of Theorem \ref{localization}.
\end{proof}

Thus, $\operatorname{Ider}_k(R)$ constitutes a coherent sheaf $\operatorname{Ider}_k(\mathcal{O}_X)$ on a $k$-scheme $X$ essentially of finite type. Using this sheaf, we can state Theorem \ref{infty integrable thm} in terms of schemes.

\begin{theorem}\label{infty integrable thm schemes}
    Let $X$ be a $k$-scheme essentially of finite type. Then there exists $M>0$ such that
$$ \operatorname{Der}_k^M(\mathcal{O}_X)= \operatorname{Der}_k^{M+1}(\mathcal{O}_X)=\dots=\operatorname{Ider}_k(\mathcal{O}_X).$$
\end{theorem}

Finally, we conclude the paper by noting that the above results extend to the case of perfect fields.

\begin{theorem}
    Let $k'$ be a perfect field of positive characteristic and let $R$ be a $k'$-algebra of finite type. Then the set of leaps of $R$ as a $k'$-algebra is finite.
\end{theorem}
\begin{proof}
    This follows from Theorem \ref{eft alg thm} and the fact that leaps are stable under separable base changes \cite[Theorem 3.27]{hernandez2020behavior}.
\end{proof}

\section*{Acknowledgements} The author would like to thank Professor Keiji Oguiso and the members of Oguiso's laboratory for their interest in this work and their valuable comments and suggestions.

\printbibliography

\end{document}